\documentclass[a4paper, reqno, 12pt]{amsart}

\usepackage{geometry}       

\usepackage{float}
\usepackage{bm}
\usepackage{fullpage,xcolor}
\usepackage[mathscr]{euscript}
\usepackage[all]{xy}
\usepackage{epsfig}
\usepackage{amsfonts}
\usepackage{mathptmx}

\usepackage{graphicx}
\usepackage{amssymb} 
\usepackage{amsmath}
\usepackage{amsthm}
\usepackage{mathrsfs}
\usepackage{epstopdf}
\usepackage{url}
\usepackage{tikz}

\usepackage{color}
\makeatletter

\@addtoreset{equation}{section}
\makeatother 

\theoremstyle{plain}
\newtheorem{theorem}{Theorem}[section]

\newtheorem{proposition}[theorem]{Proposition}

\theoremstyle{definition}
\newtheorem{definition}[theorem]{Definition}
\newtheorem{example}[theorem]{Example}
\newtheorem{remark}[theorem]{Remark}

\usepackage{latexsym}
 
\title[From annular to toroidal pseudo knots ]
 {From annular to toroidal pseudo knots }

\author{Ioannis Diamantis}
\address{Department of Data Analytics and Digitalisation,
Maastricht University, School of Business and Economics,
P.O.Box 616, 6200 MD, Maastricht,
The Netherlands.}
\email{i.diamantis@maastrichtuniversity.nl}

\author{Sofia Lambropoulou}
\address{School of Applied Mathematical and Physical Sciences, National Technical University of Athens, Zografou campus, GR-15780 Athens, Greece.}
\email{sofia@math.ntua.gr}
\urladdr{http://www.math.ntua.gr/~sofia}

\author{Sonia Mahmoudi}
\address{Advanced Institute for Materials Research, Tohoku University, 2-1-1 Katahira, Aoba-ku, Sendai 980-8577, Japan; RIKEN iTHEMS, 2-1 Hirosawa, Wako, Saitama 351-0198, Japan}
\email{sonia.mahmoudi@tohoku.ac.jp}



\date{}

\thanks{}

\begin{document}

\setcounter{section}{-1}

\begin{abstract}
In this paper, we extend the theory of planar pseudo knots to the theories of annular and toroidal pseudo knots. 
Pseudo knots are defined as equivalence classes under Reidemeister-like moves of knot diagrams characterized by crossings with undefined over/under information.
In the theories of annular and toroidal pseudo knots we introduce their respective lifts to the solid and the thickened torus. Then, we interlink these theories by representing annular and toroidal pseudo knots as planar ${\rm O}$-mixed and ${\rm H}$-mixed pseudo links.
We also explore the inclusion relations between planar, annular and toroidal pseudo knots, as well as of ${\rm O}$-mixed and ${\rm H}$-mixed pseudo links.
Finally, we extend the planar weighted resolution set to annular and toroidal pseudo knots, defining new invariants for classifying pseudo knots and links in the solid and in the thickened torus. 
\end{abstract}

\maketitle


\section{Introduction}\label{sec:0}

Knots and links have been a central focus in topology, with significant implications across various fields, including biology, chemistry, and physics. In classical knot theory, knots are typically studied through their projections on a plane, where crossings are assigned over/under information to fully capture the knot's structure. However, there are instances where this information is either unavailable or intentionally omitted. This leads to the concept of pseudo knots, which are projections of knots with some crossings left undetermined.

Pseudo knot diagrams  were introduced by Hanaki in \cite{H} as knot projections on the 2-sphere with certain double points lacking over/under information, the \textit{precrossings} or \textit{pseudo crossings}. This model is motivated by the study of DNA knots, where distinguishing over and under crossings is often unfeasible, even with electron microscopy. The theory of pseudo knots is the study of equivalence classes of pseudo diagrams under certain local combinatorial moves which extend the well-known Reidemeister moves for classical knots by taking into account also the precrossings (cf. \cite{HJMR} for further details).  

In this paper, we extend the theory of pseudo knots beyond their classical planar setting by considering them on two additional surfaces: the annulus and the torus. These surfaces are of particular interest because of their inherent rotational and reflectional symmetries, while maintaining  direct connections to the plane and between them, due to \textit{inclusion relations}. 

We study annular and toroidal pseudo knots by exploiting the theory of planar pseudo knots and by establishing interconnections due to various  inclusion relations: of a disc in the annulus and in the torus and of an annulus in the torus. We further introduce the notion of the \textit{lift} of a planar, annular or toroidal pseudo knot diagram into a closed  curve with rigid pseudo crossings, in three-space, in the solid torus  (cf. also \cite{D}) or in the thickened torus, respectively. We also introduce the notion of isotopy for the lift of a pseudo knot, establishing bijections of the corresponding theories with the planar, annular and toroidal pseudo knot theories (Proposition~\ref{prop:planarlift}, Theorem~\ref{isopST}, and Theorem~\ref{isoptt}): 
\bigbreak

\noindent {\bf Theorem 1.} 
\textit{Two  pseudo links in the three-sphere, in the solid torus or in the thickened torus, respectively, are isotopic if and only if any two corresponding planar, annular or  toroidal pseudo link diagrams of theirs, projected onto the plane, the annulus or the outer toroidal boundary, respectively, are pseudo Reidemeister equivalent. }
\bigbreak

Thanks to the lifts of planar, annular and toroidal pseudo knot diagrams we exploit further inclusion relations: of a three-ball in the solid torus and in the thickened torus, of a solid torus in the thickened torus, and of the solid torus and the thickened torus in a three-ball. The inclusion relations reflect the inherent symmetries of the geometrical objects. The notions of lifts for annular and  toroidal pseudo knots enable us to establish further connections of these theories to \textit{mixed link theories}: in the first case to ${\rm O}$-mixed pseudo links in $S^{3}$, which are pseudo links that contain a point-wise fixed unknotted component representing the complementary solid torus  (cf. also \cite{D}), in the second case to ${\rm H}$-mixed pseudo links in $S^{3}$, which are pseudo links that contain a point-wise fixed Hopf link as a sublink, whose complement is a  thickened torus (Theorem~\ref{isopmixed} and Theorem~\ref{isopmixedtt}): 

\bigbreak
\noindent {\bf Theorem 2.} 
\textit{Isotopy classes of  pseudo links in the solid torus, resp. in the thickened torus are in bijection with isotopy classes of  ${\rm O}$-mixed pseudo links resp. ${\rm H}$-mixed pseudo links in $S^{3}$, via isotopies that keep ${\rm O}$ resp. ${\rm H}$ point-wise fixed.}
\bigbreak

Based on the above, we further translate the spatial isotopies at the diagrammatic level through generalized Reidemeister theorems for ${\rm O}$-mixed pseudo links (Theorem–\ref{Omixedreid}) resp. ${\rm H}$-mixed pseudo links and relate to the corresponding pseudo Reidemeister equivalences (Theorems–\ref{Omixedreid} and~\ref{Hmixedreid}).
 Finally, we define the invariant \textit{weighted resolution sets} (WeRe sets) for annular and toroidal pseudo knots, which are sets of associated knots in the solid and thickened torus, respectively, obtained by assigning to each pseudo crossing  over or under information. Similarly, the ${\rm O}$-WeRe set and ${\rm H}$-WeRe set are also introduced for ${\rm O}$-mixed and ${\rm H}$-mixed pseudo links, which are sets of associated ${\rm O}$-mixed and ${\rm H}$-mixed links, respectively. Subsequently we have (Theorems~\ref{th:annularWeRe} and~\ref{th:toroidalWeRe}):

\bigbreak
\noindent {\bf Theorem 3.} 
\textit{The annular resp. toroidal WeRe set is an invariant of annular resp. toroidal pseudo links. Subsequently, any invariant of links in the solid torus (resp. of links in the thickened torus), applied on the elements of an  annular WeRe set (resp. of a toroidal  WeRe set), induces also an invariant set of the annular  pseudo link (resp. of the toroidal  pseudo link).  Similarly,  the ${\rm O}$-WeRe set resp. the   ${\rm H}$-WeRe set is an invariant of ${\rm O}$-mixed pseudo links resp.  ${\rm H}$-mixed pseudo links. Subsequently, any invariant  of ${\rm O}$-mixed links (resp.  of ${\rm H}$-mixed links), applied on the elements of  an ${\rm O}$-WeRe set (resp. of  an ${\rm H}$- WeRe set), induces also an invariant set of the ${\rm O}$-mixed pseudo link (resp. of the ${\rm H}$-mixed pseudo link).}
\bigbreak

Apart from the interest in studying annular and toroidal pseudo knots per se, another motivation for us is their relation to periodic tangle diagrams in a ribbon and in the plane, respectively,  through imposing one and two periodic boundary conditions by means of corresponding covering maps. For further details in periodic tangles, cf. for example \cite{DLM1,DLM2,DLM3} and references therein. 

\smallbreak

The paper is organized as follows. We begin, in \S~\ref{sec:prel}, by recalling the basic notions associated with planar pseudo knots, including the pseudo Reidemeister equivalence, the lift in three-dimensional space and the WeRe set. 

In \S~\ref{sec:annular} we extend the theory from planar to annular pseudo knots. We first present the annular pseudo Reidemeister equivalence. We then define the lifts of  annular pseudo knots in the solid torus and their isotopy moves, which lead to their representation as planar ${\rm O}$-mixed pseudo links. We also explore the inclusion relations between annular, planar and ${\rm O}$-mixed pseudo links. The section is concluded with the extension of the WeRe set for annular pseudo knots, providing a new invariant for their classification. 

In \S~\ref{sec:toroidal} we introduce the theory of toroidal pseudo knots.  We first define the  pseudo Reidemeister equivalence moves for toroidal pseudo knots. We then define the lift of toroidal pseudo knots into a closed pseudo curve in the thickened torus and the notion of  isotopy for such curves. The lift of toroidal pseudo knots leads to their representation as planar ${\rm H}$-mixed pseudo links. We also explore the inclusion relations between toroidal, annular and planar pseudo knots, as well as of ${\rm O}$-mixed and ${\rm H}$-mixed pseudo links. We conclude this section with the extension of the WeRe set for toroidal pseudo knots, defining new invariants for classifying pseudo knots and links in the thickened torus. 

The transition from planar to annular and then to toroidal setting introduces increasing complexity  due to the topological properties and symmetries of the ambient spaces. This study provides deeper insights into the interactions among planar, annular and toroidal pseudo knots.


\section{Classical Pseudo Knots}\label{sec:prel}

The classical by now theory of pseudo knots and links was introduced by Hanaki in \cite{H} and the mathematical background of pseudo knot theory was established in \cite{HJMR}. In this section, we recall the basics of pseudo knots and links in the plane, their equivalence relation and the weighted-resolution invariant set (WeRe set). We also introduce the notion of lift of a planar pseudo link in the 3-space. 

\subsection{} \label{planarpk}
A {\it planar pseudo knot/link diagram} or simply {\it pseudo knot/link diagram} consists in a regular knot or link diagram in the plane where some crossing information may be missing, that is, it is unknown which strand passes over and which strand passes under the other. These undetermined crossings are called {\it precrossings} or {\it pseudo crossings} and are depicted as transversal intersections of arcs of the diagram enclosed in a light gray circle (for an illustration see Figure~\ref{pk1}). Assigning an orientation to  each component of a pseudo link diagram results in an {\it oriented } pseudo link diagram. 

\begin{figure}[H]
\begin{center}
\includegraphics[width=1in]{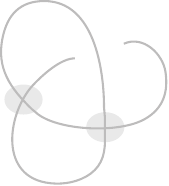}
\end{center}
\caption{A pseudo knot.}
\label{pk1}
\end{figure}

A {\it planar pseudo link} is defined as an equivalence class of pseudo link diagrams under planar isotopy and all versions of the classical Reidemeister moves $R_1, R_2, R_3$ and the {\it pseudo Reidemeister moves} $PR_1, PR_2, PR_3, PR_3^\prime$, as exemplified in Figure~\ref{reid}, all together comprising the {\it Reidemeister equivalence for pseudo links}. For an oriented Reidemeister equivalence we require also orientations to be preserved, via the oriented versions of the moves comprising the Reidemeister equivalence.

\begin{figure}[ht]
\begin{center}
\includegraphics[width=6.2in]{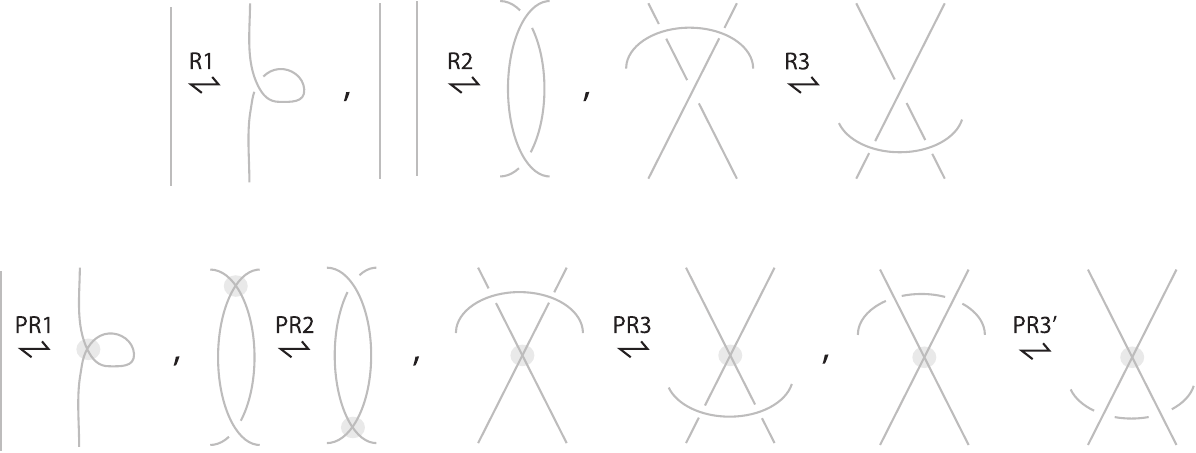}
\end{center}
\caption{Reidemeister equivalence moves for pseudo link diagrams.}
\label{reid}
\end{figure}

\begin{remark}
Consider a pseudo link diagram $K$ with no precrossings. Then, $K$ can be viewed as a classical link diagram and classical Reidemeister equivalence is preserved by pseudo link equivalence. Therefore, there is an injection of classical link types into the set of pseudo link types. 
\end{remark}

\begin{remark}
As mentioned in \cite{BJW}, pseudo links are closely related to {\it singular links}, that is, links that contain a finite number of rigid self-intersections. In particular, there exists a bijection $f$ from the set of singular link diagrams to the set of pseudo link diagrams where singular crossings are mapped to precrossings. In that way we may also recover all of the pseudo Reidemeister moves, with the exception of the pseudo Reidemeister I move (PR1). Hence, $f$ induces an onto map from the set of singular links to the set of pseudo links, since the images of two equivalent singular link diagrams are also equivalent pseudo link diagrams with exactly the same sequence of corresponding pseudo Reidemeister moves, and every pseudo link type is clearly covered.
\end{remark}


\subsection{The lift of a planar pseudo link in 3-space}

We now introduce the lift of a pseudo link diagram as (a collection of) curve(s) in three-dimensional space whose regular projections are pseudo link diagrams.

\begin{definition} \label{spatiallift}
 The {\it lift} of a planar pseudo link diagram in three-dimensional space is defined as follows: every classical crossing is embedded in a sufficiently small 3-ball so that the over arc is embedded in its upper boundary and the under arc is embedded in its lower boundary, while the precrossings are supported by sufficiently small rigid discs, which are embedded in three-space.  By lifting the precrossings in rigid discs in three-space we preserve the pseudo link's essential structure while respecting the `ambiguity' of its precrossings. The simple arcs connecting crossings can be also replaced by isotopic ones in three-space. The resulting lift, called  {\it spatial pseudo links}, is a collection of closed curve(s) in three-space consisting in embedded discs from which emanate embedded arcs.    
\end{definition}

\begin{figure}[ht]
\begin{center}
\includegraphics[width=3in]{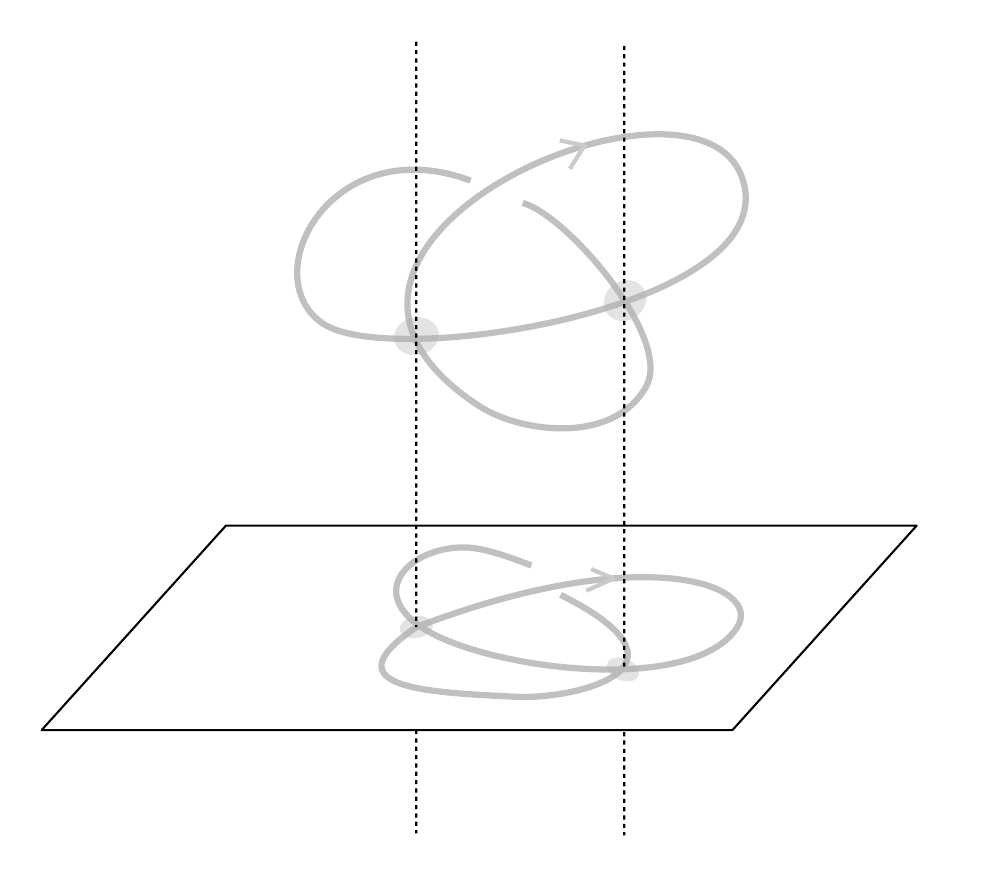}
\end{center}
\caption{The lift of a pseudo knot diagram to a spatial pseudo knot.}
\label{trefoil-lifting}
\end{figure}

Clearly, any regular projection of a spatial pseudo link, whereby a disc supporting a pseudo crossing does not project on an arc, is a planar pseudo link diagram. 

\begin{definition}
Two (oriented) spatial pseudo links are said to be {\it isotopic}
 if they  are related  by  arc and disc isotopies. 
\end{definition}

In the above context,  one can easily derive the following.

\begin{proposition}\label{prop:planarlift}
Two  (oriented) spatial pseudo links are  isotopic if and only if any two corresponding planar pseudo link diagrams of theirs are (oriented) Reidemeister equivalent. 
\end{proposition}


\subsection{The weighted resolution set of pseudo links}

In this subsection we recall an invariant of pseudo links, the weighted resolution set, through their  resolution sets \cite{HJMR}. A {\it resolution} of a pseudo link diagram \( K \) is a specific assignment of a crossing  (over  or under) for every precrossing in \( K \), for an illustration see Figure~\ref{fig:weretr}. We then have the following:

\begin{definition}\label{wrs}\rm
The {\it weighted resolution set} or {\it WeRe set}, of a planar pseudo link diagram \( K \) is a collection of ordered pairs \((K_i, p_{K_i})\), where \( K_i \) represents a resolution of \( K \), and \( p_{K_i} \) denotes the probability of obtaining from \( K \) the equivalence class of \( K_i \)  through a random assignment of crossing types, with equal likelihood for positive and negative crossings.
\end{definition}

It can be easily confirmed that the WeRe set is preserved under the equivalence moves for pseudo links \cite{HJMR}. Hence, we have:

\begin{theorem}\label{wereinv}
The WeRe set is an invariant of planar pseudo links. Subsequently, any classical invariant applied on the elements of the  WeRe set induces also an invariant set of the planar pseudo link.
\end{theorem}

\begin{example}\label{ex1}
Consider the pseudo trefoil knot of Figure~\ref{pk1}. In Figure~\ref{fig:weretr} we illustrate the resolution set of this pseudo knot resulting in the following WeRe set: 
$$
\big{\{}({\rm trefoil\ knot, 1/4}), \, ({\rm unknot}, 3/4) \big{\}}
$$
Further, applying the Jones (or, equivalently, the normalized Kauffman bracket) polynomial for classical knots to the WeRe set we obtain the invariant set: $\big{\{}( t + t^3 - t^4, 1/4), \, (1, 3/4) \big{\}}$.

\begin{figure}[H]
    \centering
    \includegraphics[width=0.6\linewidth]{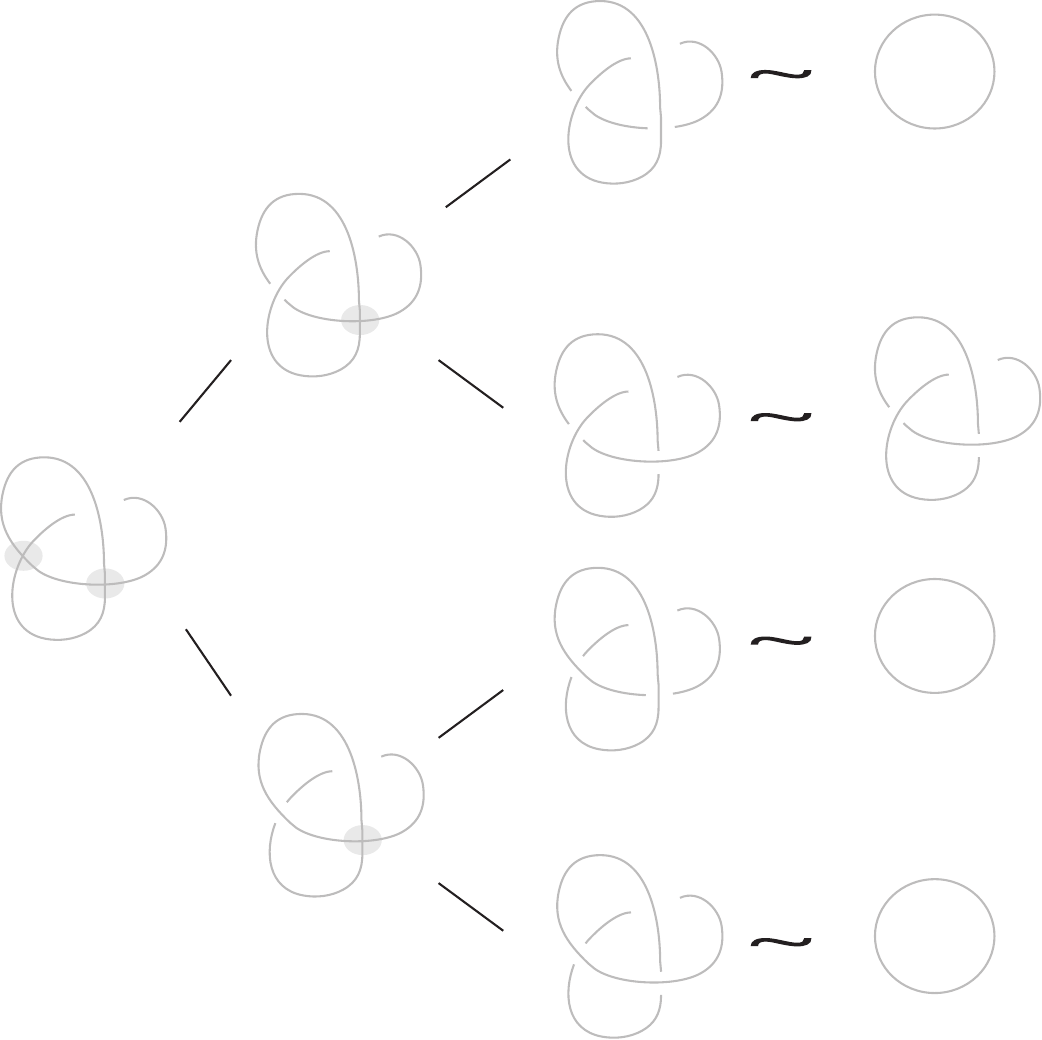}
    \caption{The resolution set of a pseudo knot.}
    \label{fig:weretr}
\end{figure}
\end{example}

\begin{remark} \label{differentprobabilities}
Consider the pseudo trefoil knot again, but this time 
 let us modify the resolution probabilities, assigning different probabilities to each precrossing. In this case, the resolution set will remain the same but the weighted resolution set will change, having a different probability distribution of the same classical links. Hence, we conclude that a pseudo link diagram $L$ with altered  resolution probabilities may lead to non-equivalent weighted resolution sets.
\end{remark}


\section{Annular pseudo knot theory}\label{sec:annular}

In this section we introduce and study the theory of annular pseudo knots and links, that is, pseudo link diagrams in the interior of the annulus $\mathcal{A}$, subjected to Reidemeister-like moves. The annulus $\mathcal{A}$ is the space $ S^1 \times D^1$ with two circular boundary components, and may be also represented as an once punctured disc, meaning a disc with a hole at its center. 
 We  study annular pseudo links in different geometric contexts, such as through their lifts to closed curves with precrossings in the solid torus and through their representations as planar ${\rm O}$-mixed pseudo links, that is, pseudo links with one fixed unknotted component. 
 We also define the invariant WeRe set for an annular pseudo link, which is a set of associated links in the solid torus, and the ${\rm O}$-WeRe set for an ${\rm O}$-mixed pseudo link, which is a set of associated ${\rm O}$-mixed  links. Finally, we discuss the relations of annular pseudo links with planar ones and with links in the solid torus, through inclusion relations.

\begin{definition}
An {\it annular pseudo knot/link diagram} consists in a regular knot or link diagram in the interior of the annulus $\mathcal{A}$, where some crossing information may be missing, in the sense that it is unknown which strand passes over the other. As in the classical case, these undetermined crossings are called {\it precrossings} or {\it pseudo crossings} and are depicted as transversal intersections of arcs of the diagram enclosed in a light gray circle. Assigning an orientation to each component of an annular pseudo link diagram results in an {\it oriented } annular pseudo link diagram. 
\end{definition}

In Figure~\ref{pkannulus} we illustrate an annular pseudo link diagram in $\mathcal{A}$. In particular, observe that this pseudo link diagram contains two components: a null-homologous loop linked to an essential locally knotted loop winding twice along the meridian.

\begin{figure}[H]
\begin{center}
\includegraphics[width=2.4in]{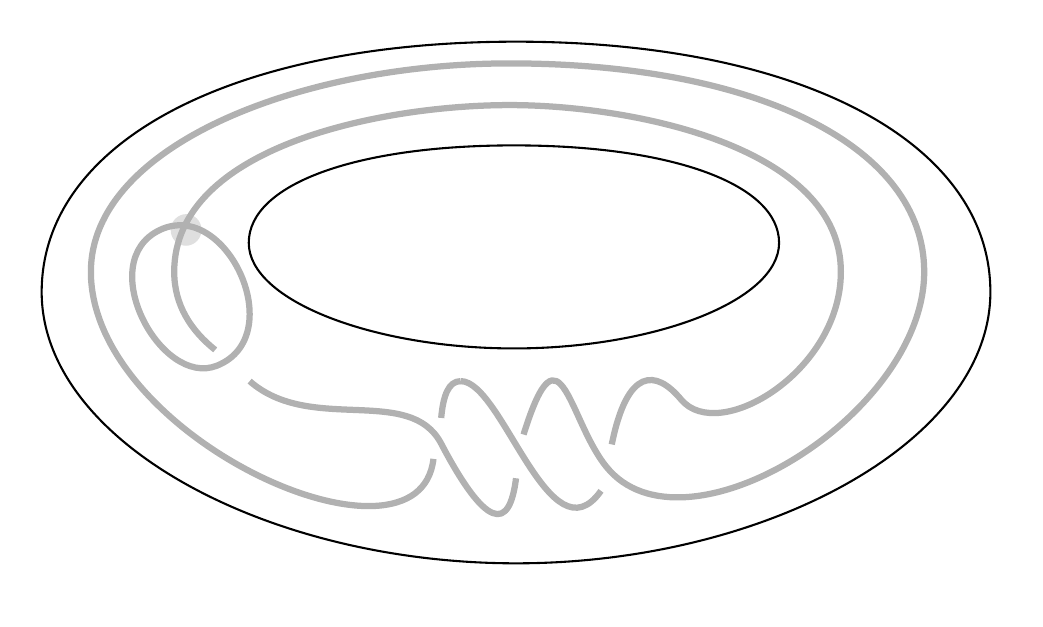}
\end{center}
\caption{An annular pseudo link.}
\label{pkannulus}
\end{figure}

\begin{definition}\label{anpl}
An {\it annular pseudo link} is defined as an equivalence class of annular pseudo link diagrams under surface isotopy and all versions of the classical Reidemeister moves and the extended pseudo Reidemeister moves, as exemplified in Figure~\ref{reid}, all together comprising the {\it Reidemeister equivalence} for annular pseudo links. As for classical pseudo links, for an oriented Reidemeister equivalence we require also orientations to be preserved, via the oriented versions of the moves comprising the Reidemeister equivalence.
\end{definition}

\begin{remark}
As in the classical case, one may relate the theory of annular pseudo links to the theory of {\it annular singular links} via a bijection from the set of annular singular link diagrams to the set of annular pseudo link diagrams. This bijection maps singular crossings to precrossings. 
 The mapping carries through to all of the pseudo Reidemeister moves, with the exception of the pseudo Reidemeister I move (PR1). Hence, we have an onto map from the set of annular singular links to the set of annular pseudo links, since the images of two equivalent annular singular link diagrams are also equivalent annular pseudo link diagrams with exactly the same sequence of corresponding pseudo Reidemeister moves, and every annular pseudo link type is clearly covered.
\end{remark}


\subsection{The lift of annular pseudo links in the solid torus}\label{sec:lift-annulus}

In this section we define the lift of annular pseudo links in the solid torus similarly to the lift of pseudo links in three-dimensional space (recall Definition~\ref{spatiallift}) but with additional constraints. We consider the thickening $\mathcal{A} \times I$, where  $I$ denotes the unit interval, $I=[0,1]$ (see Figure~\ref{annularlifting}). Note that the  space $\mathcal{A} \times I$ is homeomorphic to the solid torus, ST.

\begin{definition} \label{solidlift}
The {\it lift} of an annular pseudo link diagram in the thickened annulus $\mathcal{A} \times I$ is defined so that: each classical crossing is embedded in a sufficiently small 3-ball that lies entirely within the thickened annulus, precrossings are supported by sufficiently small rigid discs, which are embedded in the thickened annulus, and the simple arcs connecting the crossings can be replaced by isotopic ones in the thickened annulus. The lifting of the precrossings within these rigid discs  maintains the pseudo link's essential structure while preserving the `ambiguity' of its precrossings. The resulting lift is called a {\it  pseudo link in the solid torus} and it is a collection of closed curve(s) constrained to the interior of the solid torus, consisting in embedded discs from which emanate embedded arcs.
\end{definition}

\begin{figure}[ht]
\begin{center}
\includegraphics[width=3in]{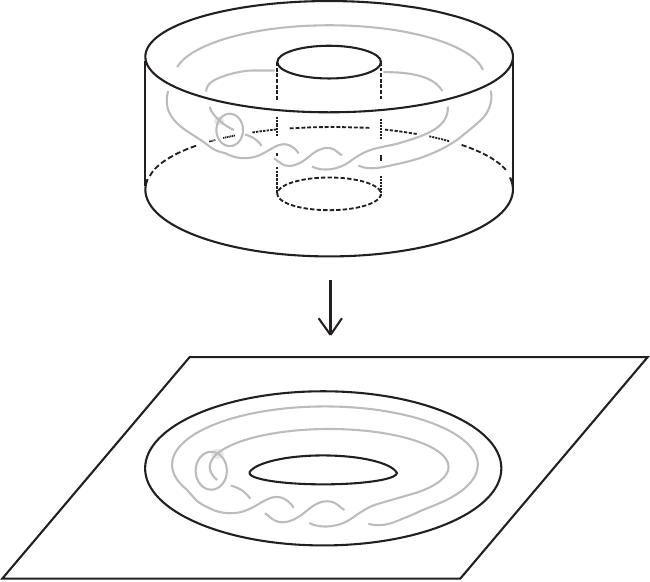}
\end{center}
\caption{The lift of an annular pseudo link to a  pseudo link in $\mathcal{A} \times I$.}
\label{annularlifting}
\end{figure}

\begin{definition} \label{solidliftisotopy}
Two (oriented)  pseudo links in the solid torus are said to be {\it isotopic} if they are related by isotopies of arcs and discs that are confined to the interior of the solid torus.
\end{definition}

In this context, the following result holds:

\begin{theorem}\label{isopST}
Two  (oriented) pseudo links in the solid torus $\mathcal{A} \times I$ are isotopic if and only if any two corresponding annular pseudo link diagrams of theirs, projected onto the annulus $\mathcal{A} \times \{0\}$, are (oriented) pseudo Reidemeister equivalent.
\end{theorem}


\subsection{The mixed pseudo link approach}\label{subsec:Omixed}

In this subsection we represent annular pseudo links as mixed pseudo link diagrams in the plane as well as spatial mixed pseudo links in three-space. 

It is established in \cite{LR1} that isotopy classes of (oriented) links in a knot/link complement correspond bijectively to isotopy classes of mixed links in the three-sphere $S^3$ or in three-space, through isotopies that preserve a fixed sublink which represents the knot/link complement. In particular, viewing ST as the complement of a solid torus in  $S^3$, an (oriented) link in ST is represented uniquely by a mixed link  in  $S^3$ that contains the standard unknot, ${\rm O}$, as a point-wise fixed sublink, which represents the complementary solid torus. Cf.  \cite{La1}.

Using now the spatial lift of a pseudo link, recall Definition~\ref{spatiallift}, one can define:

\begin{definition}
 An (oriented)  ${\rm O}$-\textit{mixed pseudo link} in $S^{3}$ is an (oriented) spatial  pseudo link ${\rm O} \cup K$  which contains the standard unknot ${\rm O}$ as a fixed sublink, and the sublink $K$ resulting by removing ${\rm O}$, as the \textit{moving part} of the mixed pseudo link, such that there are no precrossing discs between the fixed and the moving part.  
\end{definition}

Then, by the same reasoning as for (oriented) links in ST, and by Definitions~\ref{solidlift} and~\ref{solidliftisotopy}, we obtain that an (oriented) pseudo link $K$ in ST is represented uniquely by an (oriented) ${\rm O}$-mixed pseudo link ${\rm O} \cup K$ in $S^{3}$.  Namely:

\begin{theorem}\label{isopmixed}
Isotopy classes of (oriented) pseudo links in the solid torus are in bijective correspondence with isotopy classes of (oriented) ${\rm O}$-mixed pseudo links in $S^{3}$ via isotopies that keep ${\rm O}$ fixed.
\end{theorem}

\noindent For an example see Figure~\ref{pmixedtor}. The reader may compare with \cite{D}, where a pseudo link in ST is defined as an ${\rm O}$-mixed pseudo link.

\begin{figure}[H]
\begin{center}
\includegraphics[width=4in]{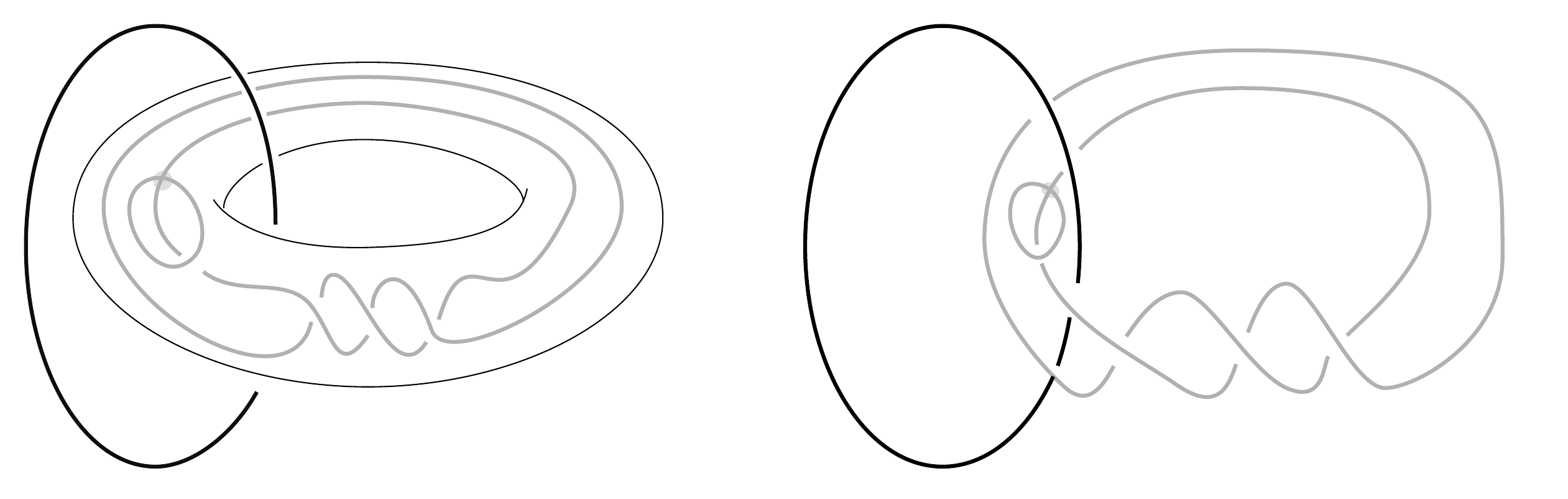}
\end{center}
\caption{A pseudo link in the solid torus and its corresponding ${\rm O}$-mixed pseudo link.}
\label{pmixedtor}
\end{figure}

Taking now a diagrammatic approach we define:

\begin{definition}
An (oriented)  ${\rm O}$-\textit{mixed pseudo link diagram} is an (oriented) regular projection ${\rm O}\cup D$ of an (oriented)  ${\rm O}$-mixed pseudo link ${\rm O}\cup K$ on the plane of ${\rm O}$, such that some double points are precrossings, as projections of the precrossings of ${\rm O}\cup K$, there are no precrossings between the fixed and the moving part, and the rest of the crossings, which are either crossings of arcs of the moving part or \textit{mixed crossings} between arcs of the moving and the fixed part, are endowed with over/under information.  
\end{definition}

Combining  the equivalence of planar pseudo link diagrams (recall Section~\ref{planarpk}) and the theory of mixed links (recall \cite{LR1,La1}) we obtain the discrete diagrammatic equivalence of ${\rm O}$-mixed pseudo links:

\begin{theorem} \label{Omixedreid}
[The ${\rm O}$-mixed pseudo Reidemeister equivalence]\label{reidplink}
Two (oriented) ${\rm O}$-mixed pseudo links in $S^{3}$ are isotopic if and only if any two (oriented)  ${\rm O}$-mixed pseudo link diagrams of theirs  differ by planar isotopies, a finite sequence of the classical and the pseudo Reidemeister moves, as exemplified in Figure~\ref{reid}, for the moving parts of the mixed pseudo links, and moves that involve the fixed and the moving parts, called {\rm mixed Reidemeister moves}, comprising the moves $MR_2, MR_3, MPR_3$, as exemplified in Figure~\ref{mpr}.
\end{theorem}

\begin{figure}[H]
\begin{center}
\includegraphics[width=5.1in]{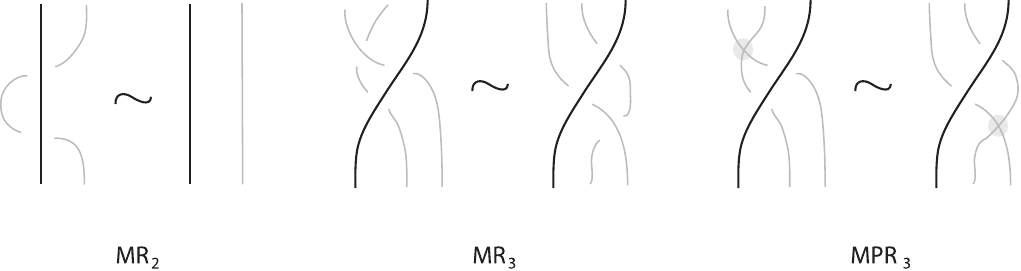}
\end{center}
\caption{The mixed Reidemeister moves for ${\rm O}$-mixed pseudo link diagrams.}
\label{mpr}
\end{figure}

In the last two subsections we lifted annular pseudo link diagrams in the thickened annulus and  related pseudo links in the solid torus to ${\rm O}$-mixed pseudo links and ${\rm O}$-mixed pseudo link diagrams. These are recapitulated in Figure~\ref{ppsttt}, where $\mathcal{A}$ is represented as an once punctured disc.

\begin{figure}[H]
\begin{center}
\includegraphics[width=4.7in]{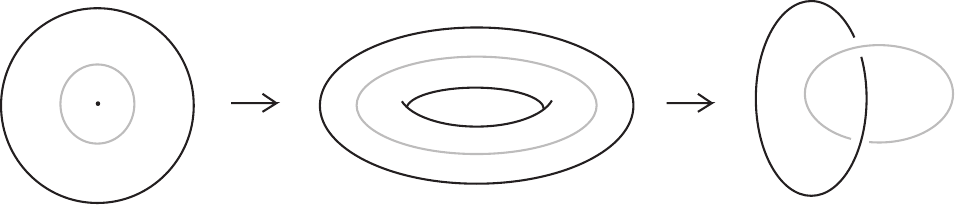}
\end{center}
\caption{The transition of a pseudo link in $\mathcal{A}$ to a pseudo link in ST, to an ${\rm O}$-mixed pseudo link.}
\label{ppsttt}
\end{figure}

Further, Theorems~\ref{isopST},~\ref{isopmixed} and~\ref{Omixedreid} culminate in the following diagrammatic equivalence:

\begin{theorem}\label{annulartoOmixed}
Two  (oriented) annular pseudo link diagrams are (oriented) pseudo Reidemeister equivalent if and only if  any two corresponding ${\rm O}$-mixed pseudo link diagrams of theirs are ${\rm O}$-mixed pseudo Reidemeister equivalent.
\end{theorem}


\subsection{Annular inclusions} \label{annularinclusions}

Annular pseudo link diagrams with no precrossings can be viewed as link diagrams in the annulus, and Reidemeister equivalence in the annulus is compatible with annular pseudo link equivalence. Thus, there is an injection of annular links into annular pseudo links. We note that annular link diagrams modulo the Reidemeister equivalence correspond bijectively to isotopy classes of links in the solid torus  ST, cf. for example \cite{Tu,La2}. 

The inclusion of a disc  in the annulus induces an injection of the theory of (planar) pseudo links into the theory of annular pseudo links. In terms of liftings, the inclusion of a three-ball  in the thickened annulus induces an injection of the theory of spatial pseudo links into the theory of pseudo links in the solid torus, from the above and by Theorem~\ref{isopST}. 

On the other hand, the inclusion of the annulus (resp. thickened annulus) in a disc (resp. three-ball) induces a surjection of the theory of annular  pseudo links (resp.  pseudo links in ST)  onto planar (resp. spatial) pseudo links, where the essential components are mapped to usual components. View Figure~\ref{planarannular}.

\begin{figure}[H] 
\begin{center} 
\includegraphics[width=6in]{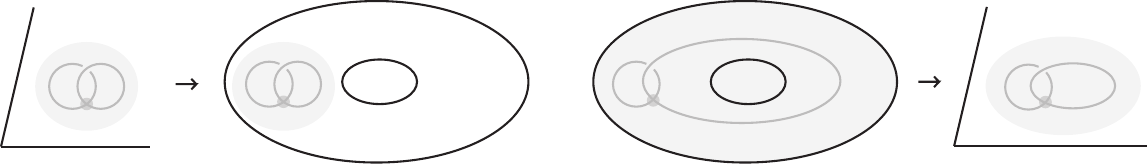} 
\end{center} 
\caption{Inclusion relations: a disc in the annulus and  the annulus in a disc.} 
\label{planarannular} 
\end{figure}

Finally, in terms of ${\rm O}$-mixed pseudo links, the inclusion of the annulus in a disc  corresponds to omitting  the fixed part ${\rm O}$, so we are left with a planar pseudo link.


\subsection{The weighted resolution set for annular pseudo links}

In this subsection we extend to annular pseudo links the weighted resolution set defined in Definition~\ref{wrs} for planar pseudo links \cite{HJMR},  and we prove that it is an invariant of theirs. Indeed:

\begin{definition}\label{resannularmixed}\rm 
 A {\it resolution} of an annular pseudo link diagram \( K \) is a specific assignment of crossing types (positive or negative) for every precrossing in \( K \).  The result is a link diagram in the annulus, which lifts to a link in the solid torus, recall Subsection~\ref{annularinclusions}. 
  Similarly, an {\it ${\rm O}$-resolution} of the ${\rm O}$-mixed pseudo link diagram \({\rm O} \cup K \) is a specific assignment of crossing types (positive or negative) for every precrossing in \({\rm O} \cup K \). In this case the resulting link is an ${\rm O}$-mixed link in $S^3$, representing uniquely the lift of the resolution of \( K \) in ST. 
 \end{definition}

 \begin{definition}\label{wrsannularmixed}\rm 
The {\it annular weighted resolution set} or {\it annular WeRe set}, of an annular  pseudo link diagram \( K \) is a collection of ordered pairs \((K_i, p_{K_i})\), where \( K_i \) represents a resolution of \( K \), and \( p_{K_i} \) denotes the probability of obtaining from \( K \) the equivalence class of \( K_i \) through a random assignment of crossing types, with equal likelihood for positive and negative crossings. 

Similarly, the {\it ${\rm O}$-weighted resolution set} or {\it ${\rm O}$-WeRe set}, of an ${\rm O}$-mixed pseudo link diagram  \({\rm O} \cup K \)  is a collection of ordered pairs \(({\rm O} \cup K_i, p_{K_i})\), where \( {\rm O} \cup K_i \) represents a resolution of \({\rm O} \cup K \), and \( p_{K_i} \) denotes the probability of obtaining from \( {\rm O} \cup K \) the equivalence class of \( {\rm O} \cup K_i \) through a random assignment of crossing types, with equal likelihood for positive and negative crossings.
\end{definition}

Further, we have the following result:

\begin{theorem}\label{th:annularWeRe}
The annular WeRe set is an invariant of annular pseudo links. Similarly, the ${\rm O}$-WeRe set is an invariant of ${\rm O}$-mixed pseudo links. Subsequently, any  invariant of links in the solid torus resp. of ${\rm O}$-mixed  links, applied on the elements of an annular resp. an ${\rm O}$-WeRe set, induces also an invariant set of the  annular resp. the ${\rm O}$-mixed pseudo link.
\end{theorem}

\begin{proof}
It follows from Theorem~\ref{wereinv} of \cite{HJMR} and from Definition~\ref{anpl} that the annular WeRe set is, indeed, invariant under the standard Reidemeister moves $R_1, R_2, R_3$ and the pseudo Reidemeister moves $PR_1, PR_2, PR_3$ and $PR_3^{\prime}$, since these moves are compatible with the corresponding moves with no precrossings.  
Further,  an ${\rm O}$-mixed pseudo link diagram \({\rm O} \cup K \) is by definition a planar  pseudo link diagram, thus, by virtue of Theorem~\ref{wereinv}  its ${\rm O}$-weighted resolution set is an invariant of \({\rm O} \cup K \). Indeed, it suffices to show that the ${\rm O}$-WeRe set is invariant under the mixed Reidemeister moves. The $MR_2$ and $MR_3$ moves preserve the resolution set, since they do not involve precrossings. The $MPR_3$ moves are similar to the $MR_3$ moves, since regardless of which resolution is considered for the precrossing, the result is an $MR_3$ move that does not change the knot type. 

 Moreover, any invariant of links in the solid torus (e.g. \cite{HK,Tu,La2}) applied to the elements of an annular resp. ${\rm O}$-WeRe set will respect the same set of local equivalence moves, hence will preserve the WeRe set, so it induces  an invariant set of the  annular resp. ${\rm O}$-mixed pseudo link. 
\end{proof}

\begin{example}
In this example we illustrate the resolution sets of an annular  pseudo trefoil with two precrossings, where  $\mathcal{A}$ is viewed as an once punctured disc. The resulting links are links in the solid torus ST, namely: an essential trefoil with probability 1/4, a  twice  counterclockwise descending essential unknot  with probability 1/4 and a twice clockwise descending essential unknot  with probability 1/2. In  $\mathcal{A}$, an {\it essential knot} is a knot diagram that cannot be contracted to a point within $\mathcal{A}$.

\begin{figure}[H]
    \centering
     \includegraphics[width=0.9\linewidth]{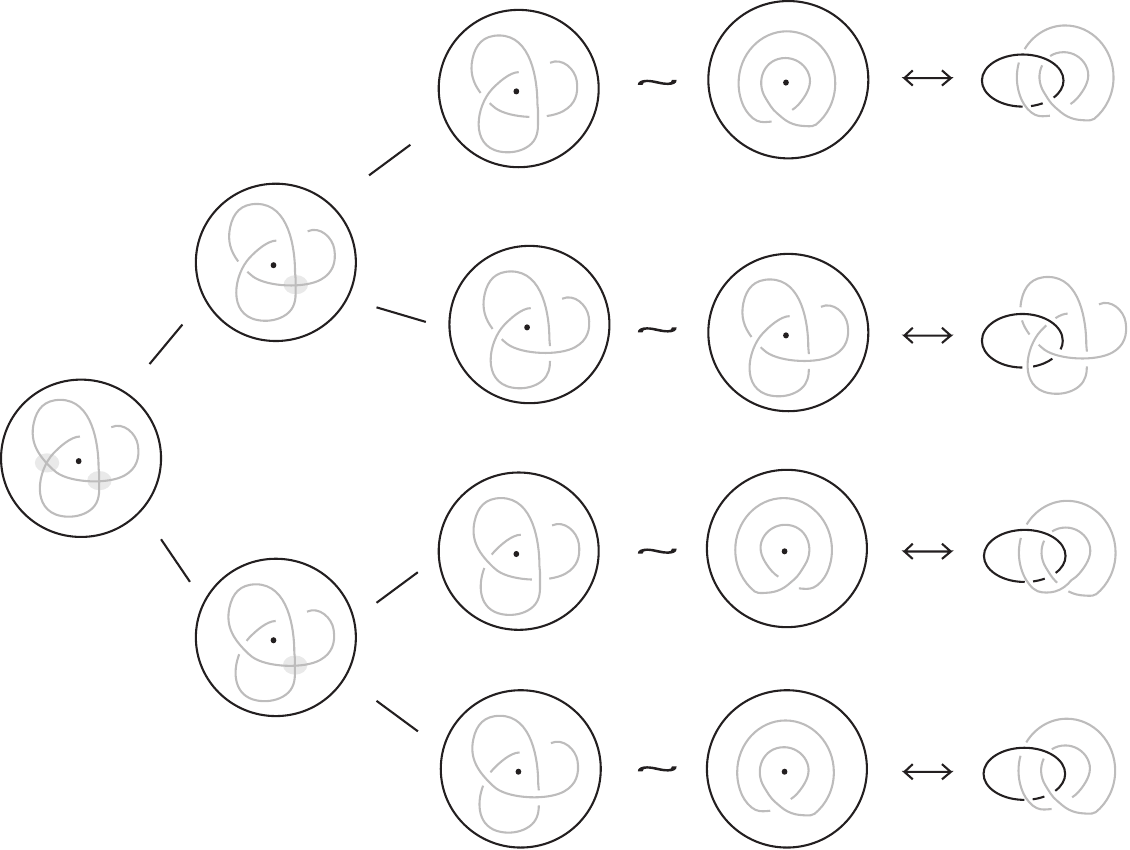}
    \caption{The resolution sets of an annular pseudo trefoil.}
    \label{fig:weretrst}
\end{figure}
\end{example}



\section{Toroidal pseudo knot theory}\label{sec:toroidal}
 
In this section, we introduce and develop the theory of pseudo knots and links in the torus, that is, pseudo link diagrams in the torus subjected to Reidemeister-like moves. We study  toroidal pseudo links in different geometric contexts, namely through their lifts to closed curves with precrossings in the thickened torus and through their representations as planar ${\rm H}$-mixed pseudo links, that is, pseudo links with one fixed sublink, the Hopf link. We also define the invariant WeRe set for a toroidal pseudo link, which is a set of associated links in the thickened torus, and the ${\rm H}$-WeRe set for an ${\rm H}$-mixed pseudo link, which is a set of associated ${\rm H}$-mixed  links. Finally, we investigate the relationships between toroidal pseudo links with links in the thickened torus as well as with annular and planar pseudo links, through illuminating inclusion relations, highlighting the interplay between these various contexts.

The surface of the torus, which is the space $\mathcal{T} = S^1 \times S^1$, can be viewed either meridian-wise, as the circular gluing of two cylinders along their boundary circles, or longitude-wise, as the gluing of two annuli along their outer and inner  boundary circles. Equivalently, as the identification space of a cylinder, by identifying the two circular boundary components, or even, as the identification space of an annulus, by identifying the inner and the outer circular boundary components. 

\begin{definition}
A {\it toroidal pseudo knot/link diagram} consists in a regular knot or link diagram in the surface of a torus, where some crossing information may be missing, in the sense that it is not known which strand passes over the other. These undetermined crossings are the {\it precrossings} or {\it pseudo crossings} and are depicted as transversal intersections of arcs of the diagram enclosed in a light gray circle. For an example view Figure~\ref{pkttorus}. Assigning an orientation to each component of an toroidal pseudo link diagram results in an {\it oriented } toroidal pseudo link diagram. 
\end{definition}

\begin{figure}[H]
\begin{center}
\includegraphics[width=2.3in]{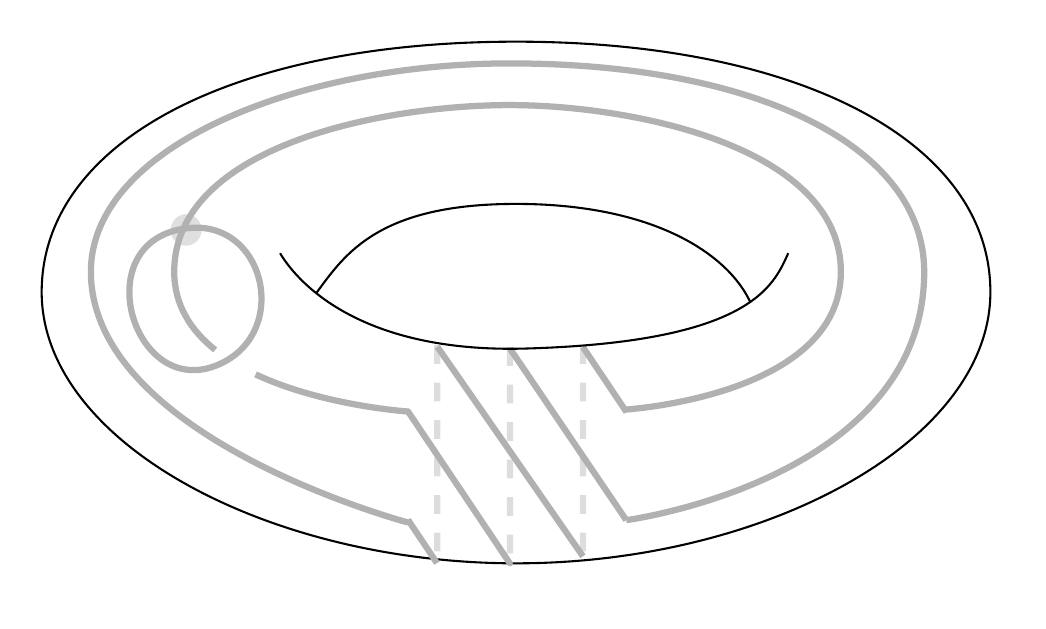}
\end{center}
\caption{A toroidal pseudo link.}
\label{pkttorus}
\end{figure}

Due to the topology of the toroidal surface, note that toroidal pseudo link diagrams have two types of essential loops that are not homologically trivial: the longitudinal ones, as in the case of annular pseudo link diagrams, but also the meridional ones, and, of course, combinations of these, such as the torus knots and links. For example, the pseudo link in Figure~\ref{pkttorus} has one essential component which winds twice in the longitudinal direction and thrice in the meridional direction.  We proceed with defining an equivalence relation in the set of toroidal pseudo link diagrams.
 
\begin{definition}\label{torpl}
A {\it toroidal pseudo link} is defined as an equivalence class of toroidal pseudo link diagrams under surface isotopy and all versions of the classical Reidemeister moves and the extended pseudo Reidemeister moves, as exemplified in Figure~\ref{reid}, all together comprising the {\it Reidemeister equivalence} for toroidal pseudo links. As for classical and annular pseudo links, for an oriented Reidemeister equivalence we require also orientations to be preserved via the oriented versions of the moves.
\end{definition}

\begin{remark}\label{singt}
As in the classical and annular case, the theory of toroidal pseudo links can be related to the theory of {\it toroidal singular links}. Indeed, there is a bijection from the set of toroidal singular link diagrams to the set of toroidal pseudo link diagrams. This bijection maps singular crossings to precrossings. 
 The mapping carries through to all of the pseudo Reidemeister moves, with the exception of the pseudo Reidemeister I move ($PR_1$). Hence, we obtain an onto map from the set of toroidal singular links to the set of toroidal pseudo links, since the images of two equivalent toroidal singular link diagrams are also equivalent toroidal pseudo link diagrams with exactly the same sequence of corresponding pseudo Reidemeister moves, and every toroidal pseudo link type is clearly covered.
\end{remark}


\subsection{The lift of toroidal pseudo links in the thickened torus}\label{sec:lift-toroidal}

In this subsection we define the lift of toroidal pseudo links in the thickened torus in analogy to the lift of annular pseudo links in the solid torus (recall Definition~\ref{solidlift}). We consider the thickening $\mathcal{T} \times I$, where  $I$ denotes the unit interval, $I=[0,1]$ (see Figure~\ref{pthtor}). 

Recall that a thickened torus can be viewed as a solid torus having another solid torus been drilled out from its interior. Equivalently, as the identification space of a thickened cylinder, by identifying the two annular boundary components, as illustrated in Figure~\ref{pthtor}, or even, as the gluing of two thickened annuli along their outer and their inner annular boundaries. Finally, a thickened torus can be defined as the complement in the three-sphere $S^3$ of the Hopf link. 

\begin{definition} \label{ttoruslift}
The {\it lift} of a toroidal pseudo link diagram in the thickened torus $\mathcal{T} \times I$ is defined so that: each classical crossing is embedded in a sufficiently small 3-ball that lies entirely within the thickened torus, precrossings are supported by sufficiently small rigid discs, which are embedded in the thickened torus, and the simple arcs connecting the crossings can be replaced by isotopic ones in the thickened torus. The lifting of the precrossings within these rigid discs maintains the pseudo link's essential structure while preserving the `ambiguity' of its precrossings. The resulting lift is called a {\it  pseudo link in the thickened torus} and it is a collection of closed curve(s) constrained to the interior of the thickened torus, consisting in embedded discs from which emanate embedded arcs.
\end{definition}

In Figure~\ref{pthtor} we illustrate the lift of a pseudo link in $\mathcal{T} \times I$, viewed  as the identification space of a thickened cylinder. In particular, observe that this pseudo link contains two components: a null-homologous loop linked to an essential locally knotted loop winding twice along the meridian and once along the longitude.

\begin{figure}[H]
\begin{center}
\includegraphics[width=2.9in]{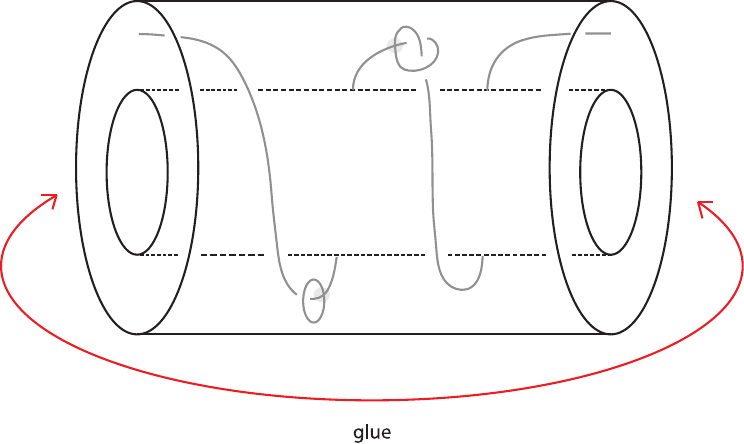}
\end{center}
\caption{The lift of a toroidal pseudo link to a pseudo link in $\mathcal{T} \times I$.}
\label{pthtor}
\end{figure}

\begin{definition} \label{ttliftisotopy}
Two (oriented) pseudo links in the thickened torus are said to be {\it isotopic} if they are related by isotopies of arcs and discs within the interior of the thickened torus.
\end{definition}

In this context, and in analogy to the spatial and annular pseudo links, the following  holds:

\begin{theorem}\label{isoptt}
Two  (oriented) pseudo links in the thickened torus $\mathcal{T} \times I$ are isotopic if and only if any two corresponding toroidal pseudo link diagrams of theirs, projected onto the outer toroidal boundary $S^1 \times S^1$, are (oriented) pseudo Reidemeister equivalent.
\end{theorem}


\subsection{Toroidal pseudo links as mixed pseudo links}

In this subsection we view the thickened torus as being homeomorphic to the complement  of the Hopf link, ${\rm H}$, in the three-sphere $S^3$, see middle illustration of Figure~\ref{tor}. Since we fix our torus $\mathcal{T}$ and subsequently  the thickened torus $\mathcal{T} \times I$, it is crucial that in this theory we have the components of ${\rm H}$ marked, say with $m$ and $l$, as depicted in Figure~\ref{tor}. Then, as in the case of annular pseudo links, an (oriented) pseudo link $K$ in $\mathcal{T} \times I$ can be represented by an (oriented) mixed link, whose  fixed part is the (marked) Hopf link ${\rm H}$, representing the thickened torus. We have the following:

\begin{definition}
 An (oriented)  ${\rm H}$-\textit{mixed pseudo link} in $S^{3}$ is an (oriented) spatial  pseudo link ${\rm H} \cup K$  which contains the marked Hopf link ${\rm H}$, as a  point-wise fixed sublink, and the sublink $K$ resulting by removing ${\rm H}$, as the \textit{moving part} of the mixed pseudo link, such that there are no precrossing discs between the fixed and the moving part.  
\end{definition}

\noindent For an example view the right-hand illustration of Figure~\ref{tor}.

\begin{figure}[H]
\begin{center}
\includegraphics[width=6in]{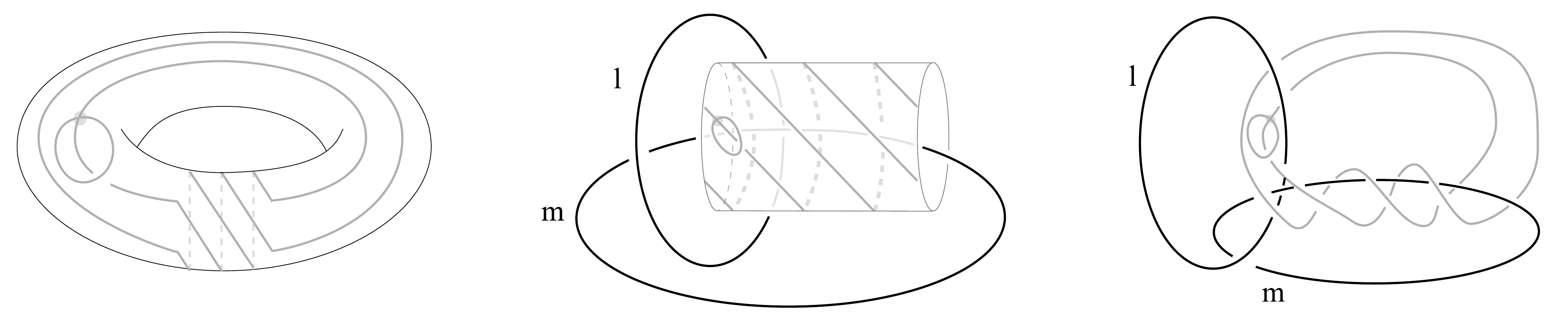}
\end{center}
\caption{A toroidal pseudo link and its corresponding ${\rm H}$-mixed pseudo link.}
\label{tor}
\end{figure}

Then, by the same reasoning as for ${\rm O}$-mixed pseudo links, we obtain the following:

\begin{theorem}\label{isopmixedtt}
Isotopy classes of (oriented) pseudo links in the thickened torus are in bijection with isotopy classes of (oriented) ${\rm H}$-mixed pseudo links in $S^{3}$, via isotopies that keep ${\rm H}$ point-wise fixed.
\end{theorem}

We now shift our focus back to the diagrammatic approach. From now on we consider ${\rm H}$ to be a fixed diagram of the Hopf link on a projection plane. 

\begin{definition}
An (oriented)  ${\rm H}$-\textit{mixed pseudo link diagram} is an (oriented) regular projection ${\rm H}\cup D$ of an (oriented)  ${\rm H}$-mixed pseudo link ${\rm H}\cup K$ on the plane of ${\rm H}$, such that some double points are precrossings, as projections of the precrossing discs of ${\rm H} \cup K$, there are no precrossings between the fixed and the moving part, and the rest of the crossings, which are either crossings of arcs of the moving part or \textit{mixed crossings} between arcs of the moving and the fixed part, are endowed with over/under information.  
\end{definition}

Consider now an isotopy of an (oriented) ${\rm H}$-mixed pseudo link ${\rm H} \cup K$ in $S^{3}$  keeping  ${\rm H}$ point-wise  fixed. Combining  the equivalence of planar pseudo link diagrams (recall Section~\ref{planarpk}) and the theory of mixed links (recall \cite{LR1}) we obtain that, in terms of ${\rm H}$-mixed pseudo link diagrams, this isotopy translates into a sequence of local moves comprising planar isotopy and the classical, pseudo  and  mixed Reidemeister moves (recall Figures~\ref{reid} and~\ref{mpr}). Note that, as supposed to the theory of ${\rm O}$-mixed pseudo links, the fixed part of the ${\rm H}$-mixed pseudo links now involves a crossing, which a moving strand can freely cross, giving rise to an extra mixed Reidemeister 3 move as illustrated in Figure~\ref{mr3}.

\begin{figure}[H]
\begin{center}
\includegraphics[width=2.6in]{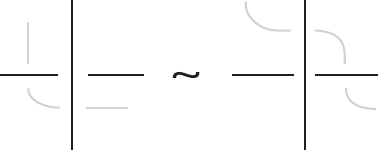}
\end{center}
\caption{A mixed $R_3$ move.}
\label{mr3}
\end{figure}

The above lead to the discrete diagrammatic equivalence of ${\rm H}$-mixed pseudo links:

\begin{theorem} \label{Hmixedreid}
[The ${\rm H}$-mixed pseudo Reidemeister equivalence]\label{reidplinktt}
Two (oriented) ${\rm H}$-mixed pseudo links in $S^{3}$ are isotopic if and only if any two (oriented)  ${\rm H}$-mixed pseudo link diagrams of theirs  differ by planar isotopies, a finite sequence of the classical and the pseudo Reidemeister moves, as exemplified in Figure~\ref{reid}, for the moving parts of the mixed pseudo links, and moves that involve the fixed and the moving parts, the {\rm mixed Reidemeister moves}, comprising the moves $MR_2, MR_3, MPR_3$, exemplified in Figures~\ref{mpr} and~\ref{mr3}.
\end{theorem}

Further, Theorems~\ref{isoptt},~\ref{isopmixedtt} and~\ref{Hmixedreid} culminate in the following diagrammatic equivalence:

\begin{theorem}\label{tortoHmixed}
Two (oriented) toroidal pseudo link diagrams are (oriented) pseudo Reidemeister equivalent if and only if  any two corresponding ${\rm H}$-mixed pseudo link diagrams of theirs are ${\rm H}$-mixed pseudo Reidemeister equivalent.
\end{theorem}

\begin{remark}\rm
If we exclude $PR_1$-moves of Figure~\ref{reid} for the moving part of a mixed pseudo link and if we change the precrossings to singular crossings (as in Remark~\ref{singt}), we obtain from Theorem\ref{tortoHmixed} the analogue of the Reidemeister theorem for toroidal singular links in terms of mixed links.
\end{remark}


\subsection{Toroidal inclusions} \label{toroidalinclusions}

 Toroidal pseudo link diagrams with no precrossings can be viewed as link diagrams in the torus, and  Reidemeister equivalence in the torus is compatible with toroidal pseudo link equivalence. Thus, there is an injection of toroidal links in toroidal pseudo links. We note that toroidal link diagrams modulo the Reidemeister equivalence correspond bijectively to isotopy classes of links in the thickened torus. So, from the lift of pseudo links (Definition~\ref{ttoruslift}) and Theorem~\ref{isoptt} we have an injection of  links in the thickened torus into pseudo links the thickened torus. 

The inclusion of a disc in the torus induces an injection of the theory of (planar) pseudo links into the theory of toroidal pseudo links. 
Further, the inclusion of the annulus in the torus induces an injection of the theory of annular pseudo links into the theory of toroidal pseudo links.  View Figure~\ref{planarannulartorus}. Similarly, from the above and by  Theorem~\ref{isoptt}, the inclusion of a three-ball in the thickened torus induces an injection of the theory of spatial pseudo links into the theory of pseudo links in the thickened torus. 

\begin{figure}[H] 
\begin{center} 
\includegraphics[width=6in]{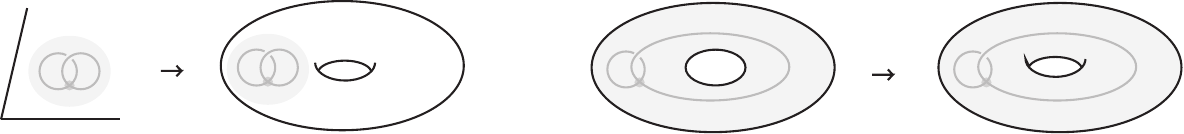} 
\end{center} 
\caption{Inclusion relations: a disc in the torus and  an annulus in the torus.} 
\label{planarannulartorus} 
\end{figure} 

Moreover, using the description of the thickened torus as the gluing of two thickened annuli and the inclusion of a thickened annulus (that is, a solid torus) in the thickened torus, the above observations and Theorem~\ref{isoptt} lead to an injection of the theory of pseudo links in the solid torus into the theory of pseudo links in the thickened torus. 
 
On the other hand, the inclusion of the thickened torus in the solid torus induces a surjection of the theory of  pseudo links in the  thickened torus onto pseudo links in the solid torus, where the meridional windings trivialize.  View Figure~\ref{inclusion-ttor-st}. In view of Theorem\ref{isoptt}, this surjection induces also a surjection of the theory of toroidal pseudo links onto the theory of annular pseudo links. 

\begin{figure}[H]
\begin{center}
\includegraphics[width=5.2in]{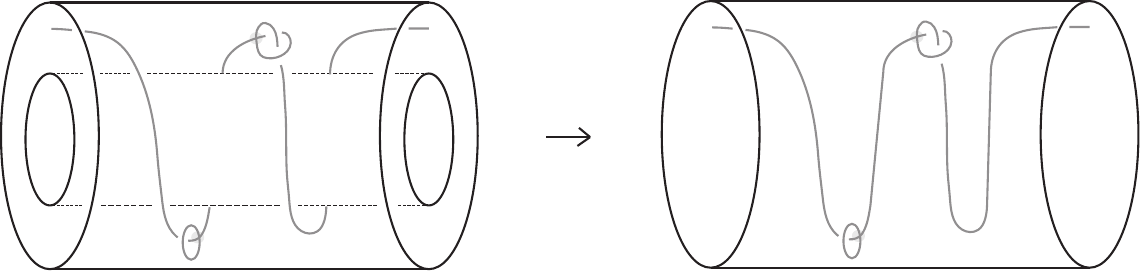}
\end{center}
\caption{Inclusion relation of the thickened torus in the solid torus.}
\label{inclusion-ttor-st}
\end{figure}

Finally,  the inclusion of the thickened torus in an enclosing 3-ball as in Figure~\ref{inclusion-tor-disk.pdf} induces a surjection of the theory of  pseudo links in the  thickened torus onto spatial pseudo links, where the meridional and longitudinal windings trivialize. By Theorem\ref{isoptt}, this surjection induces  a surjection of the theory of toroidal pseudo links onto the theory of planar pseudo links. 

\begin{figure}[H] 
\begin{center} 
\includegraphics[width=4.5in]{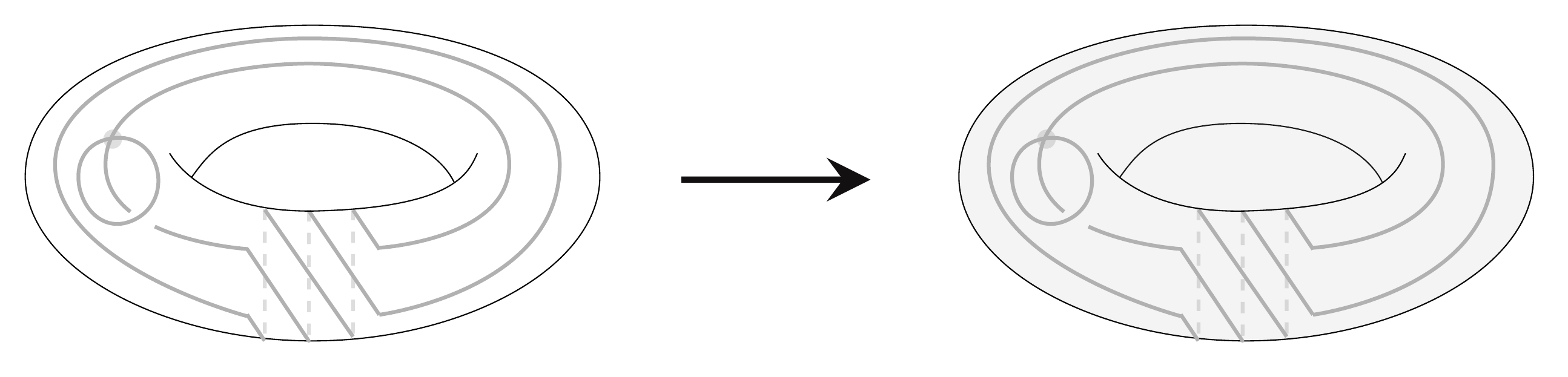} 
\end{center} 
\caption{Inclusion relation of the thickened torus in a 3-ball.} 
\label{inclusion-tor-disk.pdf} 
\end{figure} 

In terms of ${\rm H}$-mixed pseudo links, the inclusion of the thickened torus in the solid torus corresponds to omitting the fixed curve $m$, so we are left with an ${\rm O}$-mixed pseudo link. Further, the inclusion of the thickened torus in a 3-ball corresponds, in terms of ${\rm H}$-mixed pseudo links, to omitting the fixed part ${\rm H}$ entirely, so we are left with a planar pseudo link. 

We end this subsection with another remark. 

\begin{remark}
Viewing the torus as the gluing of two annuli, one  can try to project toroidal pseudo link diagrams to one of the two annular surfaces, say the upper one. 
When projecting a toroidal pseudo link diagram that cannot be isotoped within an annulus, the result is an annular pseudo link that may include an additional type of crossings, namely, {\it virtual crossings}, which are not real crossings  (cf. \cite{LK2}). This occurs because a curve wrapping around the meridian of the torus  projects in the annulus in such a way that it appears to cross another arc, even though no such crossing exists on the initial toroidal pseudo link diagram. The result is an \textit{annular virtual pseudo link diagram}. A comparative example is illustrated in Figure~\ref{tortref}, where the virtual crossing in the right-hand illustration is depicted as an encircled flat crossing.
\end{remark}

\begin{figure}[H]
\begin{center}
\includegraphics[width=4.6in]{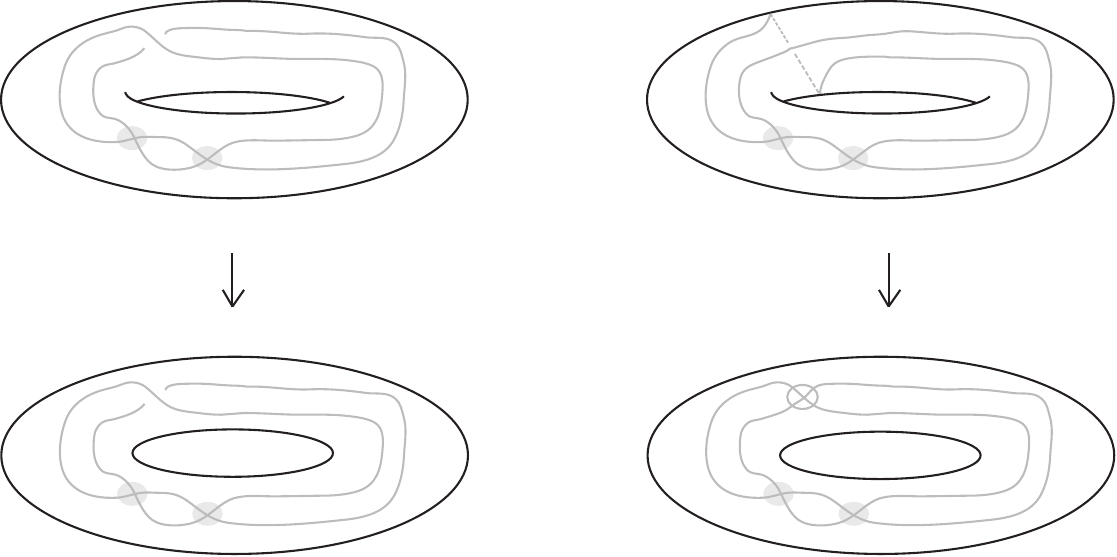}
\end{center}
\caption{Two toroidal pseudo trefoil knots projected on the annulus and the appearance of a virtual crossing.}
\label{tortref}
\end{figure}


\subsection{The weighted resolution set for toroidal pseudo links}

In this subsection we extend to toroidal pseudo links the notion of the weighted resolution set, defined first in \cite{HJMR}  for planar pseudo links (Definition~\ref{wrs}) and extended in Definition~\ref{wrsannularmixed} for annular pseudo links,  and we prove that it is an invariant of theirs. Indeed, we define:

\begin{definition}\label{restormixed}\rm 
 A {\it resolution} of a toroidal pseudo link diagram \( K \) is a specific assignment of crossing types (positive or negative) for every precrossing in \( K \). The result is a link diagram in the torus, which lifts to a link in the thickened torus. Recall Subsection~\ref{toroidalinclusions}. 
 
 Similarly, an {\it ${\rm H}$-resolution} of the ${\rm H}$-mixed pseudo link diagram \(\rm H \cup K \) is a specific assignment of crossing types (positive or negative) for every precrossing in \(\rm H \cup K \). The result is an ${\rm H}$-mixed pseudo link in $S^3$, representing uniquely the lift of the resolution of \( K \) in $\mathcal{T} \times I$.
 \end{definition}

 \begin{definition}\label{wrstormixed}\rm 
The {\it toroidal weighted resolution set} or {\it toroidal WeRe set}, of a toroidal  pseudo link diagram \( K \) is a collection of ordered pairs \((K_i, p_{K_i})\), where \( K_i \) represents a resolution of \( K \), and \( p_{K_i} \) denotes the probability of obtaining  from \( K \) the equivalence class of \( K_i \)  through a random assignment of crossing types, with equal likelihood for positive and negative crossings. 
Similarly, the {\it ${\rm H}$-weighted resolution set} or {\it ${\rm H}$-WeRe set}, of an ${\rm H}$-mixed pseudo link diagram  \({\rm H} \cup K \)  is a collection of ordered pairs \(({\rm H} \cup K_i, p_{K_i})\), where \( {\rm H} \cup K_i \) represents a resolution of \({\rm H} \cup K \), and \( p_{K_i} \) denotes the probability of obtaining from \( {\rm H} \cup K \)  the equivalence class of  \( {\rm H} \cup K_i \)  through a random assignment of crossing types, with equal likelihood for positive and negative crossings.
\end{definition}

The definitions above lead to the following:

\begin{theorem} \label{th:toroidalWeRe}
The toroidal WeRe set is an invariant of toroidal pseudo links. Similarly, the \,  ${\rm H}$-WeRe set is an invariant of \, ${\rm H}$-mixed pseudo links. Subsequently, any invariant of links in the thickened torus resp. of ${\rm H}$-mixed links, applied on the elements of a toroidal resp. an ${\rm H}$- WeRe set, induces also an invariant set of the toroidal resp. the ${\rm H}$-mixed pseudo link.
\end{theorem}

\begin{proof}
The proof follows the same approach as in the cases of planar and annular resp. ${\rm O}$-mixed pseudo link diagrams (recall Theorem~\ref{wereinv} of \cite{HJMR} and Theorem~\ref{th:annularWeRe}). Indeed, the toroidal WeRe set is invariant under the standard and pseudo Reidemeister moves of  Definition~\ref{torpl} since these moves are compatible with the corresponding moves with no precrossings. 
Further, since an ${\rm H}$-mixed pseudo link diagram \({\rm H} \cup K \) is by definition a planar pseudo link diagram, thus its ${\rm H}$-weighted resolution set is an invariant of \({\rm H} \cup K \) by Theorem~\ref{wereinv}.

Moreover, any invariant of links in the thickened torus (e.g. \cite{P, Zenkina}) applied to the elements of the toroidal WeRe set resp. ${\rm H}$-WeRe set will respect the same set of local equivalence moves, hence will preserve the WeRe sets, hence it induces an invariant set of the toroidal resp. ${\rm H}$-mixed pseudo link. 
\end{proof} 

\begin{example}
The WeRe set of the toroidal pseudo trefoil knot illustrated on the left of Figure~\ref{tortref} is the same as the WeRe set of the same pseudo trefoil knot viewed as annular. However, the illustration on the right-hand side of Figure~\ref{tortref} is more interesting as it is purely toroidal. So, in the projection on the annulus it requires the presence of a virtual crossing, recall Subsection~\ref{toroidalinclusions}. For  illustration purposes only, we present in Figure~\ref{tortreftor} the resolution set of this toroidal pseudo knot  as annular with a virtual crossing, depicted as an encircled flat crossing. Note that the isotopy moves of virtual knot theory do not apply here.

\begin{figure}[H]
\begin{center}
\includegraphics[width=5in]{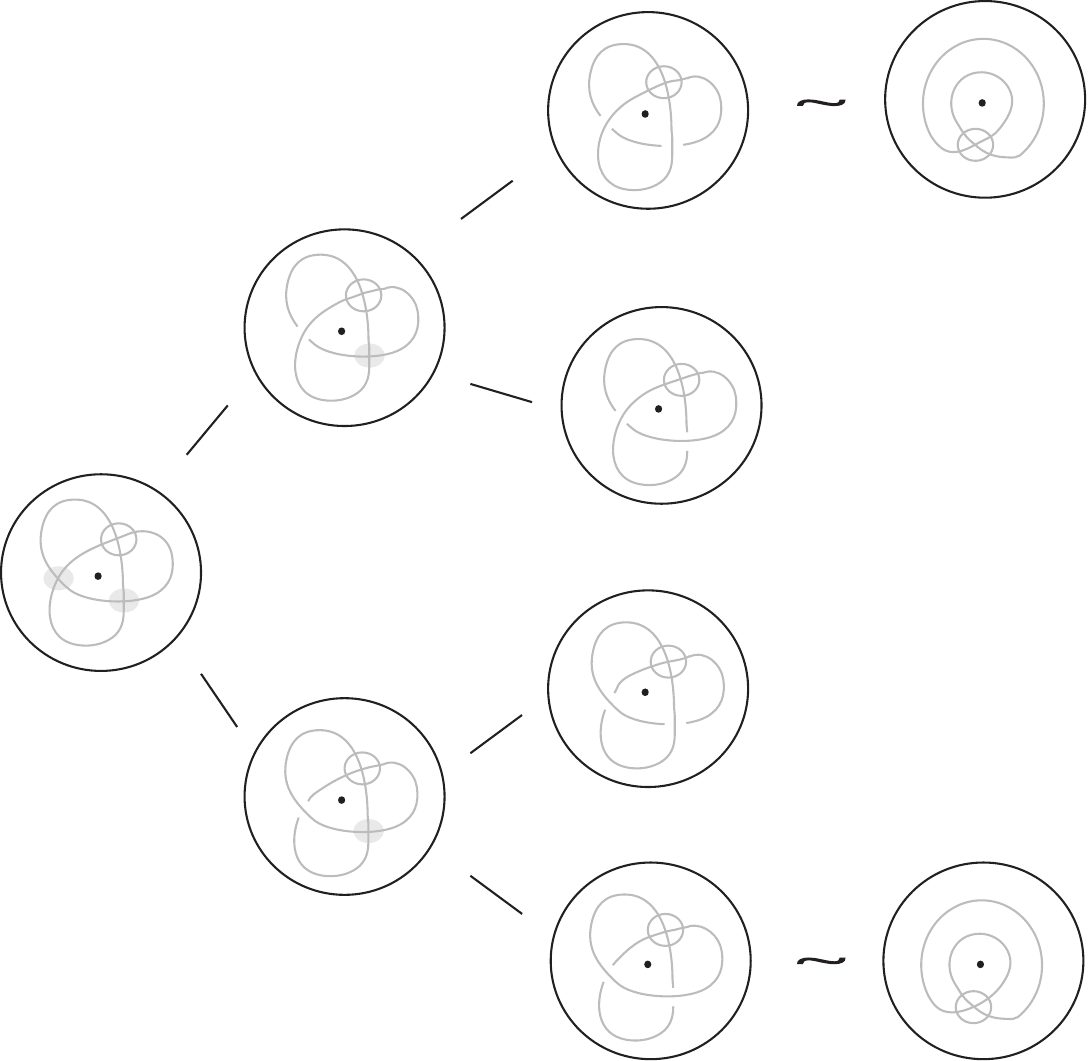}
\end{center}
\caption{The resolution set of a toroidal pseudo trefoil knot.}
\label{tortreftor}
\end{figure}
\end{example}


\bigbreak


\noindent {\bf Conclusions} 
The transition from planar to  annular and then to toroidal setting introduces increasing complexity and additional factors that must be taken into account. This study provides deeper insights into the interactions among planar,  annular and  toroidal pseudoknots and the topological properties of their ambient spaces.

\end{document}